\documentclass[final,notitlepage,8pt]{article}
\usepackage{indentfirst}
\usepackage[normalem]{ulem}
\usepackage{float}

\usepackage{graphicx}
\usepackage{enumitem}[2011/09/28]
\setenumerate{align=left, leftmargin=0pt,labelsep=.5em, labelindent=0\parindent,listparindent=\parindent,itemindent=*}
\usepackage{array}
\usepackage{empheq}
\usepackage{natbib}
\usepackage[all]{xy}

\usepackage{exscale,relsize}
\usepackage{multirow}

\usepackage{caption}[2012/02/19]

\usepackage{longtable}

\usepackage{fouriernc}
\usepackage{fourier}
\usepackage{amssymb,bm,amsmath}
\usepackage{amsfonts}

\usepackage[OT2,T1,T2A]{fontenc}

\usepackage[english,french,german,italian,russian]{babel}
\usepackage{appendix}

\DeclareMathOperator{\D}{d}
\DeclareMathOperator{\I}{Im}
\DeclareMathOperator{\R}{Re}

\bibpunct{[}{]}{,}{n}{}{;}

\def\qed{\hfill$ \blacksquare$}
\def\eor{\hfill$ \square$}

\newtheorem{theorem}{Theorem}[section]

\newtheorem{proposition}[theorem]{Proposition}
\newtheorem{corollary}[theorem]{Corollary}

\newenvironment{proof}[1][Proof]{\begin{trivlist}
\item[\hskip \labelsep {\bfseries #1}]}{\end{trivlist}}

\newenvironment{remark}[1][Remark]{\begin{trivlist}
\item[\hskip \labelsep {\bfseries #1}]}{\end{trivlist}}

\numberwithin{equation}{section}
\begin{document}

\selectlanguage{english}
\title{\textsf{On Some Integrals Over the Product of Three Legendre Functions}}\author{Yajun Zhou
\\\begin{small}\textsc{Program in Applied and Computational Mathematics, Princeton University, Princeton, NJ 08544}\end{small}
}
\date{}
\maketitle
\begin{abstract}The definite integrals $ \int_{-1}^1(1-x^2)^{(\nu-1)/2}[P_\nu(x)]^3\D x$, $ \int_{-1}^1(1-x^2)^{(\nu-1)/2}[P_\nu(x)]^2P_{\nu}(-x)\D x$, $\int_{-1}^1x(1-x^2)^{(\nu-1)/2}[P_{\nu+1}(x)]^3\D x $ and $\int_{-1}^1x(1-x^2)^{(\nu-1)/2}[P_{\nu+1}(x)]^2P_{\nu+1}(-x)\D x $ are evaluated in closed form, where $P_\nu$ is the Legendre function of degree $\nu$, and $ \R\nu>-1$. Special cases of these formulae are related to certain integrals over elliptic integrals that have arithmetic interest.\textit{Subject Classification (AMS 2010)}: 33C05, 33C10, 33C45, 33C75, 33E05,   34E05,          34E13,          44A15 \end{abstract}
\maketitle
\section{Introduction}
In this brief note, we present analytic proofs of the following integral formulae
for $ \R\nu>-1$:\begin{align}
\int_{-1}^1\frac{[P_{\nu}(x)]^3}{(1-x^{2})^{(1-\nu)/2}}\D x={}&[3-2\cos(\nu\pi)]\int_{-1}^1\frac{[P_{\nu}(x)]^2P_{\nu}(-x)}{(1-x^{2})^{(1-\nu)/2}}\D x=\frac{3-2\cos(\nu\pi)}{\pi}\left(\frac{\cos\frac{\nu\pi}{2}}{2^{\nu}}\right)^3\left[ \frac{\Gamma(\frac{1+\nu}{2})}{\Gamma(1+\frac{\nu}{2})} \right]^{4};\label{eq:P_nu_cubed_int}\\\int_{-1}^1\frac{x[P_{\nu+1}(x)]^3}{(1-x^{2})^{(1-\nu)/2}}\D x={}&[2\cos(\nu\pi)-3]\int_{-1}^1\frac{x[P_{\nu+1}(x)]^2P_{\nu+1}(-x)}{(1-x^{2})^{(1-\nu)/2}}\D x=\frac{3}{8}\frac{3-2\cos(\nu\pi)}{\pi}\left(\frac{\cos\frac{\nu\pi}{2}}{2^{\nu}}\right)^3\left[ \frac{\Gamma(\frac{1+\nu}{2})}{\Gamma(1+\frac{\nu}{2})} \right]^{4}.\label{eq:P_nu_cubed_x_int}\
\end{align}As the special cases of  Legendre functions $ P_{-1/2}$, $ P_{-1/3}=P_{-2/3}$, $ P_{-1/4}=P_{-3/4}$ and $P_{-1/6} =P_{-5/6}$  are related to the complete elliptic integrals of the first kind $ \mathbf K(k)=\int_0^{\pi/2}(1-k^2\sin^2\theta)^{-1/2}\D\theta$ \cite[Ref.][Chap.~33]{RN5}, the evaluation in Eq.~\ref{eq:P_nu_cubed_int} allows us to compute certain challenging integrals over elliptic integrals.

 While  a recent manuscript  by M. Rogers, J. G. Wan and I. J. Zucker \cite{RogersWanZucker2013preprint} was in preparation, one of the   authors (J.~G.~Wan) sent me a draft that contained a proof for the evaluation\begin{align}
\int_0^1\frac{[\mathbf K(\sqrt{1-k^2})]^3}{\sqrt{k}(1-k^2)^{3/4}}\D k={}&\frac{3[\Gamma(\frac{1}{4})]^8}{32\sqrt{2}\pi^2},
\label{eq:K'K'K'_int}\intertext{and asked me if there is an analytic verification for  another identity they discovered via numerical experimentations:}
\int_0^1\frac{[\mathbf K(\sqrt{1-k^2})]^2\mathbf K(k)}{\sqrt{k}(1-k^2)^{3/4}}\D k\overset{?}{=}{}&\frac{[\Gamma(\frac{1}{4})]^8}{32\sqrt{2}\pi^2}.
\label{eq:K'K'K_int}\end{align}
On the same day of correspondence (Feb.~21, 2013), I wrote back my deduction of Eq.~\ref{eq:K'K'K_int} from Eq.~\ref{eq:K'K'K'_int},   along with a generalization to Legendre functions of arbitrary degrees (\textit{i.e.}~the first equality in  Eq.~\ref{eq:P_nu_cubed_int}). My proof is reproduced below as Proposition~\ref{prop:App_Tricomi},  verbatim, in the form as was communicated to the authors of  \cite{RogersWanZucker2013preprint}.  Later on, I realized that one can  evaluate the integral in  Eq.~\ref{eq:K'K'K_int} without the prior knowledge of  Eq.~\ref{eq:K'K'K'_int}, drawing on the $ \nu=-1/2$ scenario of the second equality in   Eq.~\ref{eq:P_nu_cubed_int}. The computation of the integral $ \int_{-1}^1(1-x^2)^{(\nu-1)/2}[P_\nu(x)]^2P_{\nu}(-x)\D x,\R\nu>-1$ is elaborated in Proposition~\ref{prop:triple_Pnu_int}, where the connections between Eqs.~\ref{eq:P_nu_cubed_int} and \ref{eq:P_nu_cubed_x_int} are also revealed.

The proofs in  this note build on some spherical harmonic techniques, Tricomi transform identities, and Hansen-Heine type scaling analysis developed in~\cite{Zhou2013Pnu}, which are independent of the lattice sum methods in \cite{RogersWanZucker2013preprint}. In the current version of \cite{RogersWanZucker2013preprint}, the authors have announced the availability of a proof for  Eq.~\ref{eq:K'K'K_int} based on modular forms, which will appear elsewhere \cite{RogersWanZuckerInProgress2013}. Their arithmetic proof will   draw on \cite{RogersZudilin2013}   and will generalize   Eq.~\ref{eq:K'K'K_int} along another direction. Notwithstanding sharp differences in our methods and motivations, I wish to express my sincere gratitude to the authors of  \cite{RogersWanZucker2013preprint} for their inspirational work and friendly communications.   \section{An Application of Tricomi Pairing}\begin{proposition}\label{prop:App_Tricomi}We have the following evaluations:\begin{align}
\frac{[\Gamma(\frac{1}{4})]^8}{32\sqrt{2}\pi^2}={}&\frac{1}{3}\int_0^1\frac{[\mathbf K(\sqrt{1-k^2})]^3}{\sqrt{k}(1-k^2)^{3/4}}\D k=\int_0^1\frac{[\mathbf K(\sqrt{1-k^2})]^2\mathbf K(k)}{\sqrt{k}(1-k^2)^{3/4}}\D k=\int_0^1\frac{\mathbf K(\sqrt{1-k^2})[\mathbf K(k)]^2}{\sqrt{k}(1-k^2)^{3/4}}\D k\notag\\={}&\int_0^1\frac{[\mathbf K(\sqrt{1-t})]^3 }{6t^{3/4}(1-t)^{3/4}}\D t=\int_0^1\frac{[\mathbf K(\sqrt{1-t})]^2\mathbf K(\sqrt{t}) }{2t^{3/4}(1-t)^{3/4}}\D t=\int_0^1\frac{\mathbf K(\sqrt{1-t})[\mathbf K(\sqrt{t})]^{2} }{2t^{3/4}(1-t)^{3/4}}\D t\notag\\={}&\frac{\sqrt2\pi^3}{48}\int_{-1}^1\frac{[P_{-1/2}(x)]^3}{(1-x^{2})^{3/4}}\D x=\frac{\sqrt{2}\pi^{3}}{16}\int_{-1}^1\frac{[P_{-1/2}(x)]^2P_{-1/2}(-x)}{(1-x^{2})^{3/4}}\D x=\frac{\sqrt{2}\pi^{3}}{16}\int_{-1}^1\frac{P_{-1/2}(x)[P_{-1/2}(-x)]^2}{(1-x^{2})^{3/4}}\D x,\label{eq:chain_id}
\end{align}which is a special case of \begin{align}
\int_{-1}^1\frac{[P_{\nu}(x)]^3}{(1-x^{2})^{(1-\nu)/2}}\D x=[3-2\cos(\nu\pi)]\int_{-1}^1\frac{[P_{\nu}(x)]^2P_{\nu}(-x)}{(1-x^{2})^{(1-\nu)/2}}\D x,\quad \R\nu>-1,\label{eq:gen_id}
\end{align}where $ P_{\nu}(1-2z)={_2}F_1\left( \left.\begin{smallmatrix}-\nu,\nu+1\\1\end{smallmatrix}\right| z\right)$ stands for the Legendre function of the first kind of degree $\nu$. (One has  $ P_{-1/2}(x)=\frac{2}{\pi}\mathbf K(\sqrt{(1-x)/2})$ for $ -1<x\leq1$.)
\end{proposition}\begin{proof}We shall prove Eq.~\ref{eq:gen_id} for $ -1<\nu<0$.  The rest of our claims will then follow from analytic  continuation and the special case where  $\nu=-1/2$.

We first recall  the following Cauchy principal values involving fractional degree Legendre functions \cite[Ref.][Propositions~4.2 and 4.4]{Zhou2013Pnu}:\begin{align}&&\mathcal
P\int_{-1}^1\frac{2P_{\nu}(\xi)P_{\nu}(-\xi)}{\pi(x-\xi)}\D \xi={}&\frac{[P_{\nu}(x)]^2-[P_{\nu}(-x)]^2}{\sin(\nu\pi)},&&-1<x<1,\nu\in\mathbb C\smallsetminus\mathbb Z;&&\label{eq:P_nu_T}\\
&&\mathcal
P\int_{-1}^1(1+\xi)^{\nu-n} P_{\nu}(\xi)\frac{\D \xi}{2(x-\xi)}={}&(1+x)^{\nu-n} Q_{\nu}(x),&& -1<x<1,n\in\mathbb Z_{\geq0}, \R(\nu-n)>-1.&&\label{eq:P_nu_T00}
\end{align}Here, \begin{align}Q_\nu(x):=\frac{\pi}{2\sin(\nu\pi)}[\cos(\nu\pi)P_\nu(x)-P_{\nu}(-x)],\quad \nu\in\mathbb C\smallsetminus\mathbb Z;\quad Q_n(x):=\lim_{\nu\to n}Q_\nu(x),\quad n\in\mathbb Z_{\geq0}\label{eq:Qnu_defn}\end{align}defines the Legendre functions of the second kind for $ -1<x<1$. We will also need a familiar integral formula \cite[Ref.][Eq.~11.336]{KingVol1}:\begin{align}&&
\mathcal
P\int_{-1}^1\frac{(1+\xi)^{a-1}}{(1-\xi)^{a}} \frac{\D \xi}{\pi(x-\xi)}={}&\frac{(1+x)^{a-1}}{(1-x)^{a}}\cot(a\pi),&& -1<x<1,0<a<1.&&\label{eq:frac_power}
\end{align}

We  note that the Tricomi transform   $ \widehat {\mathcal T}:L^p(-1,1)\longrightarrow L^p(-1,1)$ for $ 1<p<+\infty$, defined by
\begin{align*}(\widehat{\mathcal T} f)(x):=\mathcal P\int_{-1}^1\frac{f(\xi)\D \xi}{\pi(x-\xi)},\quad\text{a.e. } x\in(-1,1)\end{align*}
satisfies a Parseval-type identity \cite[Ref.][Eq.~11.237]{KingVol1}:\begin{align}\int_{-1}^1 f(x)(\widehat{\mathcal T}g)(x)\D x+\int_{-1}^1 g(x)(\widehat{\mathcal T}f)(x)\D x={}&0,\label{eq:Tricomi_Parseval1}\end{align}and the Hardy-Poincar\'e-Bertrand formula \cite[Ref.][Eq.~11.52]{KingVol1}:\begin{align}
\widehat{\mathcal T}[f(\widehat{\mathcal T}g)+g(\widehat{\mathcal T}f)]=(\widehat{\mathcal T}f)(\widehat{\mathcal T}g)-fg,\label{eq:Tricomi_convolution}
\end{align} for any inputs $ f\in L^p(-1,1),p>1$; $ g\in L^q(-1,1),q>1$ and $ \frac1p+\frac1q<1$.

For $ -1<\nu<0$, we set $ n=0$ in Eq.~\ref{eq:P_nu_T00}, $ a=(1-\nu)/2$ in Eq.~\ref{eq:frac_power}, and apply Eq.~\ref{eq:Tricomi_convolution} to $ f(x)=(1+x)^\nu P_\nu(x)$, $ g(x)=(1+x)^{-(1+\nu)/2}/(1-x)^{(1-\nu)/2}$, which results in the following identity:\begin{align}
\int_{-1}^1\frac{P_{\nu}(\xi)\cot\frac{(1-\nu)\pi}{2}+\frac{2}{\pi}Q_\nu(\xi)}{(1-\xi^{2})^{(1-\nu)/2}}\frac{\D \xi}{\pi(x-\xi)}=\frac{\frac{2}{\pi}Q_\nu(x)\cot\frac{(1-\nu)\pi}{2}-P_\nu(x)}{(1-x^{2})^{(1-\nu)/2}},\quad -1<x<1.\label{eq:PQ_balanced_Tricomi}
\end{align}We may pair up Eqs.~\ref{eq:P_nu_T} and \ref{eq:PQ_balanced_Tricomi} in an application of the Parseval-type identity (Eq.~\ref{eq:Tricomi_Parseval1}):\begin{align}
\int_{-1}^1\frac{P_{\nu}(x)\cot\frac{(1-\nu)\pi}{2}+\frac{2}{\pi}Q_\nu(x)}{(1-x^{2})^{(1-\nu)/2}}\frac{[P_{\nu}(x)]^2-[P_{\nu}(-x)]^2}{\sin(\nu\pi)}\D x=-2\int_{-1}^1\frac{\frac{2}{\pi}Q_\nu(x)\cot\frac{(1-\nu)\pi}{2}-P_\nu(x)}{(1-x^{2})^{(1-\nu)/2}}P_{\nu}(x)P_{\nu}(- x)\D x.\label{eq:PPQ_id}
\end{align}As we spell out $ Q_\nu(x)$ using Eq.~\ref{eq:Qnu_defn}, we may  reduce Eq.~\ref{eq:PPQ_id} into a vanishing identity\begin{align*}\int_{-1}^1\frac{P_{\nu}(x)+P_{\nu}(-x)}{(1-x^{2})^{(1-\nu)/2}}\left\{ [P_\nu(x)]^2-2[2-\cos(\nu\pi)] P_\nu(x)P_\nu(-x)+[P_{\nu}(-x)]^{2}\right\}\D x=0.\end{align*}After we exploit the invariance of the factor $ (1-x^2)^{(1-\nu)/2}$ under the reflection $ x\mapsto-x$, we may deduce Eq.~\ref{eq:gen_id} from the equation above.\qed
\end{proof}

   \section{Some Integrals Over the Product of Three Legendre Functions} \begin{proposition}\label{prop:triple_Pnu_int}\begin{enumerate}[label=\emph{(\alph*)}, ref=(\alph*), widest=a]\item For $ \R\nu>-1$, we have the integral identities:\begin{align}
\int_{-1}^1\frac{[P_{\nu}(x)]^2P_{\nu}(-x)}{(1-x^{2})^{(1-\nu)/2}}\D x={}&-\frac{8}{3}\int_{-1}^1\frac{x[P_{\nu+1}(x)]^2P_{\nu+1}(-x)}{(1-x^{2})^{(1-\nu)/2}}\D x,\label{eq:raise1a}\\\int_{-1}^1\frac{[P_{\nu}(x)]^3}{(1-x^{2})^{(1-\nu)/2}}\D x={}&\frac{8}{3}\int_{-1}^1\frac{x[P_{\nu+1}(x)]^3}{(1-x^{2})^{(1-\nu)/2}}\D x,\label{eq:raise1b}
\end{align}and the recursion relation:\begin{align}
\int_{-1}^1\frac{[P_{\nu}(x)]^2P_{\nu}(-x)}{(1-x^{2})^{(1-\nu)/2}}\D x={}&-\frac{64(\nu+2)^{4}}{(\nu+1)^{4}}\int_{-1}^1\frac{[P_{\nu+2}(x)]^2P_{\nu+2}(-x)}{(1-x^{2})^{-(1+\nu)/2}}\D x.\label{eq:Phi_L_rec}
\end{align}\item We have  the following evaluation:\begin{align}
\int_{-1}^1\frac{[P_{\nu}(x)]^2P_{\nu}(-x)}{(1-x^{2})^{(1-\nu)/2}}\D x=\frac{1}{\pi}\left(\frac{\cos\frac{\nu\pi}{2}}{2^{\nu}}\right)^3\left[ \frac{\Gamma(\frac{1+\nu}{2})}{\Gamma(1+\frac{\nu}{2})} \right]^{4},\quad \R\nu>-1.\label{eq:triple_Pnu_exotic_int}
\end{align}\end{enumerate}\end{proposition} \begin{proof}\begin{enumerate}[label=(\alph*),widest=a]\item Using the standard recursion relations for Legendre functions, one can directly verify that\begin{align*}\frac{\D}{\D x}[(1-x^2)^{-(\nu+1)/2}P_{\nu+1}(x)]={}&(\nu+1)(1-x^2)^{-(\nu+3)/2}P_{\nu}(x),\notag\\\frac{\D}{\D x}[(1-x^2)^{\nu+1}P_{\nu}(x)P_{\nu}(-x)]={}&(\nu+1)(1-x^2)^{\nu}[P_{\nu}(x)P_{\nu+1}(-x)-P_{\nu}(-x)P_{\nu+1}(x)],\end{align*}so an integration by parts leads to the following identity:\begin{align}\varPhi_L(\nu):={}&\int_{-1}^1(1-x^2)^{(\nu-1)/2}[P_{\nu}(x)]^2P_{\nu}(-x)\D x=-\int_{-1}^1(1-x^2)^{(\nu-1)/2}P_{\nu+1}(x)[P_{\nu}(x)P_{\nu+1}(-x)-P_{\nu}(-x)P_{\nu+1}(x)]\D x\notag\\={}&-\int_{-1}^1(1-x^2)^{(\nu-1)/2}P_{\nu}(x)\{P_{\nu+1}(x)P_{\nu+1}(-x)-[P_{\nu+1}(-x)]^{2}\}\D x,\quad \R\nu>-1.\label{eq:int_part1}\end{align} Integrating over $ (1-x^2)^{-(\nu+3)/2}P_{\nu}(x)$ again, and exploiting the relation\begin{align*}&\frac{\D}{\D x}\{(1-x^2)^{\nu+1}P_{\nu+1}(x)P_{\nu+1}(-x)\}-\frac{\D}{\D x}\{(1-x^2)^{\nu+1}[P_{\nu+1}(-x)]^{2}\}\notag\\={}&(1-x^2)^\nu\{-2 x [P_{\nu +1}(-x)]{}^2+2 x P_{\nu +1}(x) P_{\nu +1}(-x)-2( \nu+2)  P_{\nu +2}(-x) P_{\nu +1}(-x)\notag\\&-(\nu+2)  P_{\nu +2}(x) P_{\nu +1}(-x)+(\nu+2)  P_{\nu +1}(x) P_{\nu +2}(-x)],\end{align*}we may deduce\begin{align}
\varPhi_L(\nu)={}&\frac{4}{\nu+1}\int_{-1}^1(1-x^2)^{(\nu-1)/2}x[P_{\nu+1}(x)]^2P_{\nu+1}(-x)\D x\notag\\&-\frac{2(\nu+2)}{\nu+1}\int_{-1}^1(1-x^2)^{(\nu-1)/2}P_{\nu+2}(x)\{P_{\nu+1}(x)P_{\nu+1}(-x)-[P_{\nu+1}(-x)]^{2}\}\D x,\quad \R\nu>-1\label{eq:int_part2}
\end{align} from Eq.~\ref{eq:int_part1} and a few reflection transformations $ x\mapsto -x$. Now, with the recursion relation $ (2\nu+3)x P_{\nu+1}(x)=(\nu+2)P_{\nu+2}(x)+(\nu+1)P_\nu(x)$, we may take a linear combination of   Eqs.~\ref{eq:int_part1} and \ref{eq:int_part2} that eliminates $ P_\nu$ and $ P_{\nu+2}$:\begin{align}
\varPhi_L(\nu)=-\frac{8}{3}\int_{-1}^1(1-x^2)^{(\nu-1)/2}x[P_{\nu+1}(x)]^2P_{\nu+1}(-x)\D x.\label{eq:int_part3}
\end{align}Clearly, Eq.~\ref{eq:int_part3} entails Eq.~\ref{eq:raise1a}. The proof of Eq.~\ref{eq:raise1b} is essentially similar, if not simpler:\begin{align}
[3-2\cos(\nu\pi)]\varPhi_L(\nu):={}&\int_{-1}^1(1-x^2)^{(\nu-1)/2}[P_{\nu}(x)]^3\D x=2\int_{-1}^1(1-x^2)^{(\nu-1)/2}P_{\nu}(x)[P_{\nu+1}(x)]^{2}\D x\label{eq:int_part1'}\tag{\ref{eq:int_part1}$'$}\notag\\={}&\frac{8(\nu+2)}{2\nu+5}\int_{-1}^1(1-x^2)^{(\nu-1)/2}P_{\nu+2}(x)[P_{\nu+1}(x)]^{2}\D x\tag{\ref{eq:int_part2}$'$}\label{eq:int_part2'}\\={}&\frac{8}{3}\int_{-1}^1(1-x^2)^{(\nu-1)/2}x[P_{\nu+1}(x)]^3\D x.\tag{\ref{eq:int_part3}$'$}\label{eq:int_part3'}
\end{align}

From Eq.~\ref{eq:int_part2'}, we may use integration by parts to compute\begin{align}[3-2\cos(\nu\pi)]\varPhi_L(\nu)={}&\frac{4}{2\nu+5}\int_{-1}^1(1-x^2)^{(3\nu+5)/2}P_{\nu+1}(x)\frac{\D}{\D x}\{(1-x^2)^{-\nu-2}[P_{\nu+2}(x)]^{2}\}\D x\notag\\={}&\frac{4}{2\nu+5}\int_{-1}^1(1-x^2)^{(\nu-1)/2}[P_{\nu+2}(x)]^{2}[(2\nu+3)xP_{\nu+1}(x)+(\nu+2)P_{\nu+2}(x)]\D x,\label{eq:int_part4}\end{align}while the same method also brings us\begin{align}
&2(\nu+2)\int_{-1}^1(1-x^2)^{(\nu-1)/2}x[P_{\nu+2}(x)]^{2}P_{\nu+1}(x)\D x=\int_{-1}^1(1-x^2)^{(3\nu+5)/2}xP_{\nu+2}(x)\frac{\D}{\D x}\{(1-x^2)^{-\nu-2}[P_{\nu+2}(x)]^{2}\}\D x\notag\\={}&-(\nu+2)\int_{-1}^1(1-x^2)^{(\nu-1)/2}x[P_{\nu+2}(x)]^{2}P_{\nu+1}(x)\D x-\int_{-1}^1(1-x^2)^{(\nu-1)/2}[1-4(\nu+2)x^{2}][P_{\nu+2}(x)]^{3}\D x.\label{eq:int_part5}
\end{align}Therefore, a combination of Eqs.~\ref{eq:int_part4} and \ref{eq:int_part5} results in an identity valid for $ \R\nu>-1$:\begin{align}
[3-2\cos(\nu\pi)]\varPhi_L(\nu)=\frac{4(11 \nu ^2+38 \nu +33)}{3(\nu+2)(2\nu+5)}\int_{-1}^1(1-x^2)^{(\nu-1)/2}[P_{\nu+2}(x)]^{3}\D x-\frac{16(2\nu+3)}{3(2\nu+5)}[3-2\cos(\nu\pi)]\varPhi_L(\nu+2).\label{eq:int_part6}
\end{align}

Meanwhile, using the Legendre differential equation for $ P_\nu(x)$ and integration by parts, one can establish the following vanishing identity:{\begin{align}0={}&\int_{-1}^1P_{\nu}(x)\frac{\D}{\D x}\left\{ (1-x^2) \frac{\D}{\D x}\{(1-x^2)^{(\nu+1)/2}[P_{\nu+2}(x)]^{2}\}\right\}\D x+\nu(\nu+1)\int_{-1}^1(1-x^2)^{(\nu+1)/2}P_{\nu}(x)[P_{\nu+2}(x)]^{2}\D x\notag\\&+\frac{2(2\nu+3)}{\nu+1}\int_{-1}^1(1-x^2)\frac{\D P_{\nu}(x)}{\D x}\frac{\D}{\D x}\{(1-x^2)^{(\nu+1)/2}[P_{\nu+2}(x)]^{2}\}\D x-2\nu(2\nu+3)\int_{-1}^1(1-x^2)^{(\nu+1)/2}P_{\nu}(x)[P_{\nu+2}(x)]^{2}\D x\notag\\={}&2 (\nu +2) \int_{-1}^1(1-x^2)^{(\nu-1)/2} [P_{\nu +1}(x)]{}^2 [(\nu +2) P_{\nu }(x)-2 (2 \nu +3) P_{\nu +2}(x)]\D x\notag\\&+\frac{1}{\nu+1}\int_{-1}^1(1-x^2)^{(\nu-1)/2} [P_{\nu +2}(x)]{}^2 [(2 \nu +3) (\nu^{2}-3)x P_{\nu +1}(x)+(\nu +2)(5 \nu ^2+16 \nu +13) P_{\nu +2}(x)]\D x.\label{eq:Legendre_eq_van}\end{align}}Here, in the last step of Eq.~\ref{eq:Legendre_eq_van}, we have just spelt out the integrands literally, relying on   the basic recursion relation  $ (2\mu+1) x P_\mu(x)=(\mu+1)P_{\mu+1}(x)+\mu P_{\mu-1}(x)$ wherever necessary. With the aid of Eqs.~\ref{eq:int_part1'}, \ref{eq:int_part2'}  and  \ref{eq:int_part5}, we may recast  Eq.~\ref{eq:Legendre_eq_van} into the following identity for $ \R\nu>-1$:\begin{align}
[3-2\cos(\nu\pi)]\varPhi_L(\nu)={}&\frac{2 (23 \nu ^3+111 \nu ^2+177 \nu +93)}{3 (\nu +2) (2 \nu ^2+8 \nu +7)}\int_{-1}^1(1-x^2)^{(\nu-1)/2}[P_{\nu+2}(x)]^{3}\D x\notag\\&+\frac{8(2 \nu +3) (3-\nu ^2)}{3(\nu +1) (2 \nu ^2+8 \nu +7)}[3-2\cos(\nu\pi)]\varPhi_L(\nu+2).\label{eq:int_part7}
\end{align}

One can now  eliminate the expression $ \int_{-1}^1(1-x^2)^{(\nu-1)/2}[P_{\nu+2}(x)]^{3}\D x$ from Eqs.~\ref{eq:int_part6} and \ref{eq:int_part7}, so as to verify the recursion relation \begin{align*}\varPhi_L(\nu)={}&-\frac{64(\nu+2)^{4}}{(\nu+1)^{4}}\varPhi_L(\nu+2),\quad \R\nu>-1,\end{align*}as stated in Eq.~\ref{eq:Phi_L_rec}.

 \item We denote the left- and right-hand sides of Eq.~\ref{eq:triple_Pnu_exotic_int} by $ \varPhi_L(\nu)$ and $ \varPhi_R(\nu)$, respectively. We shall show that the functions $ \varPhi_L(\nu)$ and $ \varPhi_R(\nu)$ share enough  common characteristics to warrant the truthfulness of the  identity in  Eq.~\ref{eq:triple_Pnu_exotic_int}.

Firstly, we point out that both sides of  Eq.~\ref{eq:triple_Pnu_exotic_int} agree on non-negative even integers  $ \varPhi_L(2m)=\varPhi_R(2m),m\in\mathbb Z_{\geq0}$, and    $ \varPhi_L(\nu)=O((\nu-n)^3)$ for each positive odd integer $n=2m+1,m\in\mathbb Z_{\geq0}$, so that $ \varPhi_{L}(\nu)/\varPhi_R{(\nu)}$ is bounded as $\nu$ approaches any positive odd integer.   Here, one may verify $ \varPhi_L(0)=\varPhi_R(0)=\pi$ through direct computation, and the recursion relation for $ \varPhi_L(\nu)$ (Eq.~\ref{eq:Phi_L_rec}) consequently brings us   $ \varPhi_L(2m)=\varPhi_R(2m),m\in\mathbb Z_{\geq0}$. To show that   $ \varPhi_L(\nu)=O((\nu-n)^3)$ for $n=2m+1,m\in\mathbb Z_{\geq0}$, we compute  the first- and second-order derivatives of $ \varPhi_L(\nu)$ at $ \nu=n$. Bearing in mind that $ P_n(-x)=-P_n(x)$ when $n$ is a positive odd integer, we may differentiate Eq.~\ref{eq:gen_id} in $\nu$: \begin{align*}\left.\frac{\partial}{\partial\nu}\right|_{\nu=n}\int_{-1}^1\frac{[P_{\nu}(x)]^3}{(1-x^{2})^{(1-\nu)/2}}\D x=\left.\frac{\partial}{\partial\nu}\right|_{\nu=n}\left\{[3-2\cos(\nu\pi)]\int_{-1}^1\frac{[P_{\nu}(x)]^2P_{\nu}(-x)}{(1-x^{2})^{(1-\nu)/2}}\D x\right\}\end{align*}and simplify the equation above into\begin{align}
3\int_{-1}^1\frac{[P_{n}(x)]^2}{(1-x^{2})^{(1-n)/2}}\left.\frac{\partial P_{\nu}(x)}{\partial\nu}\right|_{\nu=n}\D x=5\left\{2\int_{-1}^1\frac{P_{n}(x)P_{n}(-x)}{(1-x^{2})^{(1-n)/2}}\left.\frac{\partial P_{\nu}(x)}{\partial\nu}\right|_{\nu=n}\D x+\int_{-1}^1\frac{[P_{n}(-x)]^2}{(1-x^{2})^{(1-n)/2}}\left.\frac{\partial P_{\nu}(x)}{\partial\nu}\right|_{\nu=n}\D x\right\}.\label{eq:1st_deriv_van}
\end{align}Here, each integral in Eq.~\ref{eq:1st_deriv_van} represents the same number, up to a plus or minus sign. After rearrangement, one can see that each addend  in Eq.~\ref{eq:1st_deriv_van} indeed vanishes, which proves that  $ \varPhi_L(\nu)=O((\nu-n)^2)$. Now, as we expand the identity \begin{align*}\left.\frac{\partial^{2}}{\partial\nu^{2}}\right|_{\nu=n}\int_{-1}^1\frac{[P_{\nu}(x)]^3}{(1-x^{2})^{(1-\nu)/2}}\D x=\left.\frac{\partial^{2}}{\partial\nu^{2}}\right|_{\nu=n}\left\{[3-2\cos(\nu\pi)]\int_{-1}^1\frac{[P_{\nu}(x)]^2P_{\nu}(-x)}{(1-x^{2})^{(1-\nu)/2}}\D x\right\}=5\left.\frac{\partial^{2}}{\partial\nu^{2}}\right|_{\nu=n}\int_{-1}^1\frac{[P_{\nu}(x)]^2P_{\nu}(-x)}{(1-x^{2})^{(1-\nu)/2}}\D x \end{align*} using the Leibniz rule of differentiation, we obtain\begin{align}&\left.\frac{\partial^{2}}{\partial\nu^{2}}\right|_{\nu=n}\int_{-1}^1\frac{[P_{\nu}(x)]^3}{(1-x^{2})^{(1-\nu)/2}}\D x=
6\int_{-1}^1\frac{P_{n}(x)}{(1-x^{2})^{(1-n)/2}}\left[ \left.\frac{\partial P_{\nu}(x)}{\partial\nu}\right|_{\nu=n} \right]^{2}\D x+3\int_{-1}^1\frac{[P_{n}(x)]^{2}}{(1-x^{2})^{(1-n)/2}} \left.\frac{\partial^{2} P_{\nu}(x)}{\partial\nu^{2}}\right|_{\nu=n}\D x\notag\\={}&20\int_{-1}^1\frac{P_{n}(x)}{(1-x^{2})^{(1-n)/2}}\left.\frac{\partial P_{\nu}(x)}{\partial\nu}\right|_{\nu=n}\left.\frac{\partial P_{\nu}(-x)}{\partial\nu}\right|_{\nu=n}\D x+10\int_{-1}^1\frac{P_{n}(-x)}{(1-x^{2})^{(1-n)/2}}\left[ \left.\frac{\partial P_{\nu}(x)}{\partial\nu}\right|_{\nu=n} \right]^{2}\D x\notag\\&+10\int_{-1}^1\frac{P_{n}(x)P_{n}(-x)}{(1-x^{2})^{(1-n)/2}} \left.\frac{\partial^{2} P_{\nu}(x)}{\partial\nu^{2}}\right|_{\nu=n}\D x+5\int_{-1}^1\frac{[P_{n}(-x)]^{2}}{(1-x^{2})^{(1-n)/2}} \left.\frac{\partial^{2} P_{\nu}(x)}{\partial\nu^{2}}\right|_{\nu=n}\D x.\label{eq:2nd_deriv_van}
\end{align}    Appealing again to the symmetry $ P_n(x)=-P_n(-x)$, we see that the last two lines in Eq.~\ref{eq:2nd_deriv_van} add up  to \begin{align*}-10\int_{-1}^1\frac{P_{n}(x)}{(1-x^{2})^{(1-n)/2}}\left[ \left.\frac{\partial P_{\nu}(x)}{\partial\nu}\right|_{\nu=n} \right]^{2}\D x-5\int_{-1}^1\frac{[P_{n}(x)]^{2}}{(1-x^{2})^{(1-n)/2}} \left.\frac{\partial^{2} P_{\nu}(x)}{\partial\nu^{2}}\right|_{\nu=n}\D x=-\frac{5}{3}\left.\frac{\partial^{2}}{\partial\nu^{2}}\right|_{\nu=n}\int_{-1}^1\frac{[P_{\nu}(x)]^3}{(1-x^{2})^{(1-\nu)/2}}\D x,\end{align*}thus we may reach the result   $ \varPhi_L(\nu)=O((\nu-n)^3)$ where $n=2m+1,m\in\mathbb Z_{\geq0}$. In view of the triple zero at $ \nu=1$, we may use  the recursion relation for $ \varPhi_L(\nu)=-64(\nu+2)^{4}\varPhi_L(\nu+2)/(\nu+1)^4$ (Eq.~\ref{eq:Phi_L_rec})  to analytically continue  $ \varPhi_L(\nu)$ as a meromorphic function for $ \nu\in\mathbb C\smallsetminus\{-1,-3,-5,\dots\}$, so that  all the negative odd integers are simple poles.
Based on the facts gleaned so far, we know that the ratio  $ \varPhi_L(\nu)/\varPhi_R(\nu)$ is analytic in the whole complex  $ \nu$-plane, and $ \varPhi_L(2m)/\varPhi_R(2m)=1$ for $ m\in\mathbb Z$.

Secondly, we show that the entire function\begin{align}
f(\nu):=\frac{1}{\sin\frac{\nu \pi }{2}}\left( \frac{\varPhi_L(\nu)}{\varPhi_R(\nu)}-1 \right),\quad \nu\in\mathbb C
\end{align}has at most $ O(|\I \nu|^{3/2})$ growth rate when $ \nu$ tends to infinity in a certain manner.
More precisely, we will be concerned with $\nu$ residing on a square contour $C_N$ with vertices $ 2N-\frac{1}{2}-2iN$, $ 2N-\frac{1}{2}+2iN$,  $-2N-\frac{1}{2}+2iN $, and $-2N-\frac{1}{2}-2iN $, where $N$ is a positive integer. For $ \eta\in\mathbb R$, the conical function\begin{align*}P_{-\frac12+i\eta}(\cos\theta)={_2}F_1\left( \left.\begin{array}{c}
\frac12-i\eta,\frac12+i\eta\ \\
1 \\
\end{array}\right|\sin^2\frac{\theta}{2}\right)=1+\sum_{k=1}^\infty\left( \prod_{n=1}^k\frac{4\eta^{2}+(2n-1)^2}{4n^{2}} \right)\sin^{2k}\frac{\theta}{2},\quad 0\leq\theta<\pi\end{align*}is strictly positive and increasing in $ \theta$. This allows us to deduce the following \textit{a priori} bound estimate for $ \R\nu=-1/2$:\begin{align}
|\varPhi_L(\nu)|={}&\left\vert \int_{-1}^1\frac{[P_{\nu}(x)]^2P_{\nu}(-x)}{(1-x^{2})^{(1-\nu)/2}}\D  x\right\vert\leq\int_{-1}^1\frac{[P_{\nu}(x)]^2P_{\nu}(-x)}{(1-x^{2})^{3/4}}\D x=\int_{0}^{\pi/2}\frac{[P_\nu(\cos\theta)+P_\nu(-\cos\theta)]P_\nu(\cos\theta)P_\nu(-\cos\theta)}{\sqrt{\sin\theta}}\D\theta\notag\\\leq{}& [P_\nu(0)]^{2}\int_{0}^{\pi}\frac{P_\nu(\cos\theta)}{\sqrt{\sin\theta}}\D\theta+\int_{0}^{\pi/2}\frac{[P_\nu(-\cos\theta)]^{2}P_\nu(\cos\theta)}{\sqrt{2\theta/\pi}}\left( \frac{\pi\sin\theta}{2\theta} \right)^{3/2}\D\theta\notag\\={}&\frac{[\Gamma(\frac{1}{4})]^{2}\cos^2\frac{\nu\pi}{2}}{4\pi^{2}}\left[ \sqrt{2}+2 \cos \frac{(4\nu+1)\pi}{4} \right]\left[\frac{\Gamma(\frac{1+\nu}{2})}{\Gamma(1+\frac{\nu}{2})}\right]^{3}+\frac{\pi^{2}}{4}\int_{0}^{\pi/2}[P_\nu(-\cos\theta)]^{2}P_\nu(\cos\theta)\left( \frac{\sin\theta}{\theta} \right)^{3/2}\frac{\D\theta}{\sqrt{\theta}},
\label{eq:Phi_L_est1}\end{align}where we have computed the integral $ \int_0^\pi P_\nu(\cos\theta)/\sqrt{\sin\theta}\D \theta$ using a result of T. M. MacRobert \cite{MacRobert1939}, and quoted the standard evaluation of $ P_\nu(0)$. Next, we recall the asymptotic behavior of conical functions \cite[Ref.][pp.~473-474]{Olver1974} for $ \eta\to+\infty,\theta\in[0,\pi/2]$:\begin{align}P_{-\frac12+i\eta}(\cos\theta)\sim \sqrt{\frac{\theta}{\sin\theta}}I_0(\eta\theta),\quad P_{-\frac12+i\eta}(-\cos\theta)\sim\frac{2\cosh(\eta\pi)}{\pi} \sqrt{\frac{\theta}{\sin\theta}}K_0(\eta\theta),\label{eq:IK_asympt}\end{align}where $I_0 $ and $ K_0$ are modified Bessel functions of zeroth order. This allows us to evaluate the limit\begin{align}&\lim_{\eta\to+\infty}
\frac{1}{\eta^{3/2}\left\vert\varPhi_R\left(-\frac12+i\eta\right)\sin\frac{(2i\eta-1) \pi }{4}\right\vert}\frac{\pi^{2}}{4}\int_{0}^{\pi/2}\left[P_{-\frac12+i\eta}(-\cos\theta)\right]^{2}P_{-\frac12+i\eta}(\cos\theta)\left( \frac{\sin\theta}{\theta} \right)^{3/2}\frac{\D\theta}{\sqrt{\theta}}\notag\\={}&\lim_{\eta\to+\infty}
\frac{\cosh^{2}(\eta\pi)}{\eta^{3/2}\left\vert\varPhi_R\left(-\frac12+i\eta\right)\sin\frac{(2i\eta-1) \pi }{4}\right\vert}\int_{0}^{\pi/2}I_0(\eta\theta)[K_0(\eta\theta)]^{2}\frac{\D\theta}{\sqrt{\theta}}=\frac{\pi}{2\sqrt{2}}\int_{0}^{\infty}I_0(\theta)[K_0(\theta)]^{2}\frac{\D\theta}{\sqrt{\theta}}<+\infty.\label{eq:Phi_L_est2}
\end{align} Combining the results in Eqs.~\ref{eq:Phi_L_est1} and \ref{eq:Phi_L_est2},  we see that $ f(\nu)=O(|\I\nu|^{3/2})$ is true for $ \R\nu=-\frac{1}{2}$. As $ \varPhi_L(\nu)$ and $ \varPhi_R(\nu)$ share the same recursion relation, we have $ f(\nu)=-f(\nu+2)$. Thus, the estimate  $ f(\nu)=O(|\I\nu|^{3/2})$ holds for the two vertical sides of the  square contour $C_N$, where $ \R \nu=-2N-\frac{1}{2}$ and $ \R\nu=2N-\frac{1}{2}$. To estimate the growth bound  on the two horizontal sides of the  square contour $C_N$, we begin with the scenarios where $ \I\nu =2N>10$ and $1\leq\R\nu\leq3$. After a slight modification of Eq.~\ref{eq:Phi_L_est1}, we may put down\begin{align}|\varPhi_L(\nu)|\leq{}&\int_{0}^{\pi/2}|P_\nu(\cos\theta)|^2|P_\nu(-\cos\theta)|(\sin\theta)^{\R\nu}\D\theta+\int_{0}^{\pi/2}|P_\nu(-\cos\theta)|^{2}|P_\nu(\cos\theta)|\left( \frac{\pi\sin\theta}{2\theta} \right)^{3/2}\theta^{\R\nu}\D\theta.\label{eq:Phi_L_est3}\end{align}As  the asymptotic formulae in Eq.~\ref{eq:IK_asympt} remain valid when   $ i\eta$ is replaced with $ (2\nu+1)/2$,  we may employ a scaling argument as in  Eq.~\ref{eq:Phi_L_est2} to deduce \begin{align}\frac{1}{\sin^2(\nu\pi)}
\int_{0}^{\pi/2}|P_\nu(-\cos\theta)|^{2}|P_\nu(\cos\theta)|\left( \frac{\pi\sin\theta}{2\theta} \right)^{3/2}\theta^{\R\nu}\D\theta=O\left( \frac{1}{|\I \nu|^{\R\nu+1}} \right),\quad \text{as }\I\nu\to+\infty,1\leq\R\nu\leq3.\label{eq:Phi_L_est4}
\end{align}   Likewise, we have $\csc(\nu\pi) \int_{0}^{\pi/2}|P_\nu(-\cos\theta)|(\sin\theta)^{\R\nu}\D\theta=O(|\I \nu|^{-\R\nu-1})$. By the mean value theorem for integration, there exists an acute angle $ \alpha_\nu\in(0,\pi/2)$ such that \[P_\nu(\cos\alpha_\nu) =\sqrt{\frac{\alpha_\nu}{\sin\alpha_\nu}}O\left(\left\vert I_0\left(\frac{(2\nu+1)\alpha_\nu}{2i}\right)\right\vert\right),\quad \left\vert I_0\left(\frac{(2\nu+1)\alpha_\nu}{2i}\right)\right\vert=\frac{1}{\pi}\left\vert\int_0^\pi e^{\frac{(2\nu+1)\alpha_\nu}{2i}\cos\theta}\D\theta\right\vert\leq e^{\pi|\I\nu|/2} \] and\begin{align}\int_{0}^{\pi/2}|P_\nu(\cos\theta)|^2|P_\nu(-\cos\theta)|(\sin\theta)^{\R\nu}\D\theta=|P_\nu(\cos\alpha_\nu)|^2\int_{0}^{\pi/2}|P_\nu(-\cos\theta)|(\sin\theta)^{\R\nu}\D\theta=O\left( \frac{\sin(\nu\pi)e^{\pi|\I\nu|}}{|\I \nu|^{\R\nu+1}} \right).\label{eq:Phi_L_est5}\end{align}From Eqs.~\ref{eq:Phi_L_est3}-\ref{eq:Phi_L_est5}, one can confirm the bound estimate  $ f(\nu)=O(|\I\nu|^{3/2})$ for  $ \I\nu =2N>10$ and $1\leq\R\nu\leq3$. By complex conjugation and recursion, one can extend this result to all the points on the the two horizontal sides of the  square contour $C_N$. This completes the task stated at the beginning of the paragraph.

Lastly, by an application of Cauchy's integral formula, we know that the second derivative for   the function $ f(\nu),\nu\in\mathbb C$  vanishes everywhere:\begin{align*}f''(\nu)=\frac{1}{\pi i}\lim_{N\to\infty}\oint_{C_{N}}\frac{f(z)\D z}{(z-\nu)^{3}}=0,\quad\forall\nu\in\mathbb C,\end{align*} so $ f(\nu)$ must be an affine function $ f(\nu)=a\nu+b$ for two constants $ a,b\in\mathbb C$. However, as we have the recursion  $ f(\nu)=-f(\nu+2)$,  such an affine function must be identically zero. This eventually verifies  Eq.~\ref{eq:triple_Pnu_exotic_int}. \qed

\end{enumerate}\end{proof} \begin{remark}During the course of deriving recursion relations for $ \varPhi_L(\nu)$, we have obtained various integrals that are equal to  $ \varPhi_L(\nu)$ times an elementary function of $\nu$. A sophisticated by-product of the proof above is the following:\begin{align}&\int_{-1}^1\frac{[P_{\nu+2}(x)]^3}{(1-x^{2})^{(1-\nu)/2}}\D x=[3-2\cos(\nu\pi)]\int_{-1}^1\frac{[P_{\nu+2}(x)]^2P_{\nu+2}(-x)}{(1-x^{2})^{(1-\nu)/2}}\D x\notag\\={}&\frac{3-2\cos(\nu\pi)}{\pi}\frac{2 \nu ^3+15 \nu ^2+36 \nu +29}{16 (\nu +2)^3}\left(\frac{\cos\frac{\nu\pi}{2}}{2^{\nu}}\right)^3\left[ \frac{\Gamma(\frac{1+\nu}{2})}{\Gamma(1+\frac{\nu}{2})} \right]^{4},\quad \R\nu>-1.\label{eq:int_id_plus2}\end{align}Here, the first line in Eq.~\ref{eq:int_id_plus2} can be proved in a similar vein as Eq.~\ref{eq:gen_id}, except that one chooses $ n=2$ in the application of Eq.~\ref{eq:P_nu_T00}. The second line in  Eq.~\ref{eq:int_id_plus2} is a result of Eqs.~\ref{eq:triple_Pnu_exotic_int} and \ref{eq:int_part6}. \eor\end{remark} \begin{remark}It so happens that Eq.~\ref{eq:triple_Pnu_exotic_int} can be rewritten as\begin{align}\int_{-1}^1\frac{[P_{\nu}(x)]^2P_{\nu}(-x)}{(1-x^{2})^{(1-\nu)/2}}\D x=\left[\frac{P_\nu(0)}{2^{\nu}}\right]^3\int_{-1}^1\frac{\D x}{(1-x^2)^{(1-\nu)/2}}.\label{eq:triple_Pnu_exotic_int'}\tag{\ref{eq:triple_Pnu_exotic_int}$'$}\end{align}At the moment, we are not aware of a heuristic interpretation for Eq.~\ref{eq:triple_Pnu_exotic_int'} that is  simpler than the foregoing multi-step proof of  Eq.~\ref{eq:triple_Pnu_exotic_int}.\eor \end{remark}Combining the results from Propositions~\ref{prop:App_Tricomi} and \ref{prop:triple_Pnu_int}, we have verified Eq.~\ref{eq:P_nu_cubed_int} in its entirety. The special case where $ \nu=-1/2$ corresponds to Eqs.~\ref{eq:K'K'K'_int} and \ref{eq:K'K'K_int}. In the next corollary, we apply Ramanujan's theory of elliptic functions on alternative bases \cite[Ref.][Chap.~33]{RN5} to Legendre functions of fractional degrees $P_{-1/3}=P_{-2/3} $, $ P_{-1/4}=P_{-3/4}$ and $P_{-1/6}=P_{-5/6} $, so as to deduce closed-form evaluations of certain integrals over elliptic integrals. The ratios of gamma functions will  be simplified so that only $ \Gamma(\frac13)$, $ \Gamma(\frac14)$ and $\Gamma(\frac18)$ are retained in the final presentation of the fractions \cite[Ref.][\S54]{NielsenGamma}.\begin{corollary}\label{cor:int_ids}We have the following identities{\allowdisplaybreaks\begin{align}
\frac{3^{7/2} }{2^{19/3} }\frac{[\Gamma(\frac{1}{3})]^{12}}{\pi ^7}={}&\int_{-1}^1\frac{[P_{-1/3}(x)]^3}{(1-x^2)^{2/3}}\D x=\frac{8}{\pi^3}\int_0^1\frac{3\sqrt[3] 2(1+p+p^2)^3}{(1+2 p)^{11/6} [(1-p^2) p(2+p) ]^{1/3}}\left[ \mathbf K\left( \sqrt{\frac{p^{3}(2+p)}{1+2p}} \right) \right]^3\D p\notag\\={}&\frac{8}{\pi^3}\int_0^1\frac{\sqrt[3] 2(1+p+p^2)^3}{\sqrt{3}(1+2 p)^{11/6} [(1-p^2) p(2+p) ]^{1/3}}\left[ \mathbf K\left( \sqrt{\frac{(1+p)^{3}(1-p)}{1+2p}} \right) \right]^3\D p; \\\frac{3^{7/2} }{2^{22/3} }\frac{[\Gamma(\frac{1}{3})]^{12}}{\pi ^7}={}&\int_{-1}^1\frac{[P_{-1/3}(x)]^2P_{-1/3}(-x)}{(1-x^2)^{2/3}}\D x\notag\\={}&\frac{8}{\pi^3}\int_0^1\frac{\sqrt{3}\sqrt[3] 2(1+p+p^2)^3}{(1+2 p)^{11/6} [(1-p^2) p(2+p) ]^{1/3}}\left[ \mathbf K\left( \sqrt{\frac{p^{3}(2+p)}{1+2p}} \right) \right]^2\mathbf K\left( \sqrt{\frac{(1+p)^{3}(1-p)}{1+2p}} \right)\D p\notag\\={}&\frac{8}{\pi^3}\int_0^1\frac{\sqrt[3] 2(1+p+p^2)^3}{(1+2 p)^{11/6} [(1-p^2) p(2+p) ]^{1/3}}\left[ \mathbf K\left( \sqrt{\frac{(1+p)^{3}(1-p)}{1+2p}} \right) \right]^2\mathbf K\left( \sqrt{\frac{p^{3}(2+p)}{1+2p}} \right)\D p;\\\frac{3^{4} }{2^{13/3} }\frac{[\Gamma(\frac{1}{3})]^{12}}{\pi ^7}={}&\int_{-1}^1\frac{[P_{-1/3}(x)]^3}{(1-x^2)^{5/6}}\D x=\frac{8}{\pi^3}\int_0^1\frac{\sqrt{3}\sqrt[3] 4(1+p+p^2)^4}{(1+2 p)^{13/6} [(1-p^2) p(2+p) ]^{2/3}}\left[ \mathbf K\left( \sqrt{\frac{p^{3}(2+p)}{1+2p}} \right) \right]^3\D p\notag\\={}&\frac{8}{\pi^3}\int_0^1\frac{\sqrt[3] 4(1+p+p^2)^4}{3(1+2 p)^{13/6} [(1-p^2) p(2+p) ]^{2/3}}\left[ \mathbf K\left( \sqrt{\frac{(1+p)^{3}(1-p)}{1+2p}} \right) \right]^3\D p; \\\frac{3^{4} }{2^{19/3} }\frac{[\Gamma(\frac{1}{3})]^{12}}{\pi ^7}={}&\int_{-1}^1\frac{[P_{-1/3}(x)]^2P_{-1/3}(-x)}{(1-x^2)^{5/6}}\D x\notag\\={}&\frac{8}{\pi^3}\int_0^1\frac{\sqrt[3] 4(1+p+p^2)^4}{(1+2 p)^{13/6} [(1-p^2) p(2+p) ]^{2/3}}\left[ \mathbf K\left( \sqrt{\frac{p^{3}(2+p)}{1+2p}} \right) \right]^2\mathbf K\left( \sqrt{\frac{(1+p)^{3}(1-p)}{1+2p}} \right)\D p\notag\\={}&\frac{8}{\pi^3}\int_0^1\frac{\sqrt[3] 4(1+p+p^2)^4}{\sqrt{3}(1+2 p)^{13/6} [(1-p^2) p(2+p) ]^{2/3}}\left[ \mathbf K\left( \sqrt{\frac{(1+p)^{3}(1-p)}{1+2p}} \right) \right]^2\mathbf K\left( \sqrt{\frac{p^{3}(2+p)}{1+2p}} \right)\D p;\\\frac{\sqrt{1+\sqrt{2}} (1+2 \sqrt{2})}{2^{5/2} (2+\sqrt{2})^4 }\frac{[\Gamma (\frac{1}{8})]^8}{\pi ^3 [\Gamma(\frac{1}{4})]^4}={}& \int_{-1}^1\frac{[P_{-1/4}(x)]^3}{(1-x^2)^{5/8}}\D x=\frac{8}{\pi^{3}}\int_{0}^{1}\frac{(2-t)[\mathbf K(\sqrt{t})]^3}{2 (1-t)^{5/8} t^{1/4}}\D t=\frac{8}{\pi^{3}}\int_{0}^{1}\frac{\sqrt{2}(1+t)[\mathbf K(\sqrt{t})]^3}{ (1-t)^{1/4} t^{5/8}}\D t;\\
\frac{ (1+ \sqrt{2})^{3/2}}{2^{5/2} (2+\sqrt{2})^4 }\frac{[\Gamma (\frac{1}{8})]^8}{\pi ^3 [\Gamma(\frac{1}{4})]^4}={}& \int_{-1}^1\frac{[P_{-1/4}(x)]^2P_{-1/4}(-x)}{(1-x^2)^{5/8}}\D x=\frac{8}{\pi^{3}}\int_{0}^{1}\frac{(2-t)[\mathbf K(\sqrt{t})]^2\mathbf K(\sqrt{1-t})}{\sqrt{2} (1-t)^{5/8} t^{1/4}}\D t\notag\\={}&\frac{8}{\pi^{3}}\int_{0}^{1}\frac{(1+t)[\mathbf K(\sqrt{t})]^2\mathbf K(\sqrt{1-t})}{ (1-t)^{1/4} t^{5/8}}\D t;\\\frac{(2-\sqrt{2})^{3/2}(3+\sqrt{2})}{2^{15/4}}\frac{[\Gamma (\frac{1}{8})]^8}{\pi ^3 [\Gamma(\frac{1}{4})]^4}={}&  \int_{-1}^1\frac{[P_{-1/4}(x)]^3}{(1-x^2)^{7/8}}\D x=\frac{8}{\pi^{3}}\int_{0}^{1}\frac{(2-t)^{2}[\mathbf K(\sqrt{t})]^3}{4(1-t)^{7/8} t^{3/4}}\D t=\frac{8}{\pi^{3}}\int_{0}^{1}\frac{(1+t)^{2}[\mathbf K(\sqrt{t})]^3}{\sqrt{2}(1-t)^{3/4} t^{7/8}}\D t;\\\frac{(2-\sqrt{2})^{3/2}}{2^{15/4}}\frac{[\Gamma (\frac{1}{8})]^8}{\pi ^3 [\Gamma(\frac{1}{4})]^4}={}&\int_{-1}^1\frac{[P_{-1/4}(x)]^2P_{-1/4}(-x)}{(1-x^2)^{7/8}}\D x=\frac{8}{\pi^{3}}\int_{0}^{1}\frac{(2-t)^{2}[\mathbf K(\sqrt{t})]^2\mathbf K(\sqrt{1-t})}{2\sqrt{2}(1-t)^{7/8} t^{3/4}}\D t\notag\\={}&\frac{8}{\pi^{3}}\int_{0}^{1}\frac{(1+t)^{2}[\mathbf K(\sqrt{t})]^2\mathbf K(\sqrt{1-t})}{{2}(1-t)^{3/4} t^{7/8}}\D t; \\\frac{3^{3/2}(2-\sqrt{3})}{2^{4}}\frac{[\Gamma(\frac{1}{4})]^{8}}{\pi^{5}}={}&\int_{-1}^1\frac{[P_{-1/6}(x)]^3}{(1-x^2)^{7/12}}\D x=\frac{8}{\pi^{3}}\int_{0}^{1}\frac{3^{5/4}[\mathbf  K(\sqrt{t})]^3}{2^{5/6} [t(1-t)]^{1/6}}\D t; \\\frac{3}{2^{4}(1+\sqrt{3})}\frac{[\Gamma(\frac{1}{4})]^{8}}{\pi^{5}}={}&\int_{-1}^1\frac{[P_{-1/6}(x)]^2P_{-1/6}(-x)}{(1-x^2)^{7/12}}\D x=\frac{8}{\pi^{3}}\int_{0}^{1}\frac{3^{5/4}[\mathbf  K(\sqrt{t})]^2\mathbf K(\sqrt{1-t})}{2^{5/6} [t(1-t)]^{1/6}}\D t;\\\frac{3^{3/2}(1+\sqrt{3})^{2}}{2^3}\frac{[\Gamma(\frac{1}{4})]^{8}}{\pi^{5}}={}&\int_{-1}^1\frac{[P_{-1/6}(x)]^3}{(1-x^2)^{11/12}}\D  x=\frac{8}{\pi^{3}}\int_{0}^1\frac{3^{1/4}(1-t+t^2)[\mathbf  K(\sqrt{t})]^3}{2^{1/6} [t(1-t)]^{5/6}}\D t;\\\frac{3(1+\sqrt{3})}{2^3}\frac{[\Gamma(\frac{1}{4})]^{8}}{\pi^{5}}={}&\int_{-1}^1\frac{[P_{-1/6}(x)]^2P_{-1/6}(-x)}{(1-x^2)^{11/12}}\D x=\frac{8}{\pi^{3}}\int_{0}^1\frac{3^{1/4}(1-t+t^2)[\mathbf  K(\sqrt{t})]^2\mathbf K(\sqrt{1-t})}{2^{1/6} [t(1-t)]^{5/6}}\D t,
\end{align}}which are special cases of  Eq.~\ref{eq:P_nu_cubed_int}.  \qed\end{corollary}
\begin{remark}Thanks to the contiguous relations of Legendre functions, and the differentiation formulae for $ \mathbf K(k)=\int_0^{\pi/2}(1-k^2\sin^2\theta)^{-1/2}\D\theta$ that involve the complete elliptic integrals of the second kind  $ \mathbf E(k)=\int_0^{\pi/2}(1-k^2\sin^2\theta)^{1/2}\D\theta$, the Legendre function $ P_\nu$ is always expressible in terms of complete elliptic integrals of the first and second kinds, whenever the fractional part of $\nu$ is one of the following numbers: $ 1/2,1/3,2/3,1/4,3/4,1/6,5/6$.

Thus,  besides what have been tabulated in Corollary~\ref{cor:int_ids}, we may use Eqs.~\ref{eq:P_nu_cubed_int}  and \ref{eq:P_nu_cubed_x_int} to evaluate an infinite family of integrals over complete elliptic integrals. As a glimpse of such evaluations, we point out that setting $ \nu=1/2$ in  Eq.~\ref{eq:P_nu_cubed_int}  and $ \nu=-1/2$ in Eq.~\ref{eq:P_nu_cubed_x_int} would bring us\begin{align}
\frac{384\pi^{3}}{[\Gamma(\frac{1}{4})]^{8}}={}&\int_{-1}^1\frac{[P_{1/2}(x)]^3}{(1-x^2)^{1/4}}\D x=\frac{8}{\pi^{3}}\int_0^1\frac{\sqrt{2}[2\mathbf E(\sqrt{t})-\mathbf K(\sqrt{t})]^{3}}{ [t(1-t)]^{1/4}}\D t\notag\\={}&3\int_{-1}^1\frac{[P_{1/2}(x)]^2P_{1/2}(-x)}{(1-x^2)^{1/4}}\D x=\frac{24}{\pi^{3}}\int_0^1\frac{\sqrt{2}[2\mathbf E(\sqrt{t})-\mathbf K(\sqrt{t})]^{2}[2\mathbf E(\sqrt{1-t})-\mathbf K(\sqrt{1-t})]}{ [t(1-t)]^{1/4}}\D t;\\\frac{9[\Gamma(\frac{1}{4})]^{8}}{32\pi^{5}}={}&\int_{-1}^1\frac{x[P_{1/2}(x)]^3}{(1-x^2)^{3/4}}\D x=\frac{8}{\pi^{3}}\int_0^1\frac{(1-2t)[2\mathbf E(\sqrt{t})-\mathbf K(\sqrt{t})]^{3}}{ \sqrt{2}[t(1-t)]^{3/4}}\D t\notag\\={}&-3\int_{-1}^1\frac{x[P_{1/2}(x)]^2P_{1/2}(-x)}{(1-x^2)^{3/4}}\D x=-\frac{24}{\pi^{3}}\int_0^1\frac{(1-2t)[2\mathbf E(\sqrt{t})-\mathbf K(\sqrt{t})]^{2}[2\mathbf E(\sqrt{1-t})-\mathbf K(\sqrt{1-t})]}{ \sqrt{2}[t(1-t)]^{3/4}}\D t,
\end{align}  as one may reckon that   $ P_{1/2}(1-2t)=\frac{2}{\pi}[2\mathbf E(\sqrt{t})-\mathbf K(\sqrt{t})]$.    \eor\end{remark}\begin{remark}
It might be worth noting that the special values of gamma functions  appearing in  Corollary~\ref{cor:int_ids} have also arisen from certain lattice sums \cite{RogersWanZucker2013preprint}. It would be thus interesting to see the deduction of these integral formulae from modular forms and special values of $L$-functions.\eor\end{remark}

\noindent \textbf{Acknowledgements} The author thanks an anonymous referee for suggestions on improving the presentation of this paper. This work was partly supported by the Applied Mathematics Program within the Department of Energy (DOE) Office of Advanced Scientific Computing Research (ASCR) as part of the Collaboratory on Mathematics for Mesoscopic Modeling of Materials (CM4).
The author thanks Prof.~Weinan E (Princeton University) for his encouragements.
\bibliography{Pnu}
\bibliographystyle{unsrt}

\end{document}